\documentclass[letterpaper,11pt]{amsart}

\usepackage{epsfig}
\usepackage{amssymb}
\usepackage{amsfonts}
\usepackage{amsmath}
\usepackage{graphicx}
\usepackage{mathrsfs}
\usepackage{stmaryrd}
\usepackage{amsthm}
\usepackage{MnSymbol}
\usepackage[all]{xy}
\usepackage[top=1.25in, bottom=1.25in, left=1.25in, right=1.25in]{geometry}

\newtheorem{theorem}{Theorem}
\newtheorem{lemma}{Lemma}
\newtheorem{proposition}{Proposition}
\newtheorem{corollary}{Corollary}

\newtheorem{definition}{Definition}
\newtheorem{remark}{Remark}

\numberwithin{lemma}{section}
\numberwithin{proposition}{section}
\numberwithin{theorem}{section}
\numberwithin{corollary}{section}

\title[Scattering operators for ACHE metrics of Bergman type]{On the scattering operators for ACHE metrics of Bergman type on strictly pseudoconvex domains }
\author{Fang Wang}
\address{Shanghai Jiao Tong University, 800 Dongchuan Rd, Shanghai 200240, China.}
\email{fangwang1984@sjtu.edu.cn}
\date{January 6, 2016}
\thanks{The author was supported by Shanghai Pujiang Program No. 14PJ1405400}

\begin{document}
\maketitle

\begin{abstract}
The scattering operators associated to an ACHE metric of Bergman type on a strictly pseudoconvex domain are a one-parameter family of CR-conformally invariant pseudo-differential operators  of Heisenberg class with respect to the induced CR structure on the boundary. In this paper, we mainly show that if the boundary Webster scalar curvature is positive, then for $\gamma\in(0,1)$ the renormalised scattering operator $P_{2\gamma}$ has positive spectrum and satisfies the maximum principal; moreover,  the fractional curvature $Q_{2\gamma}$ is also positive. This is parallel to the result of Guillarmou-Qing \cite{GQ} for the real case. 
We also give two energy extension formulae for $P_{2\gamma}$, which are parallel to the energy extension given by Chang-Case \cite{CC} for the real case. 
\end{abstract}
\maketitle

\section{Introduction}

The scattering operators associated to the Laplacian operator for a \textit{real asymptotically hyperbolic}  manifold have been extensively studied, see for example \cite{Me1} \cite{GZ} \cite{JS} \cite{GQ} \cite{CC} and the references cited there. The purpose of this paper is to extend some results in \cite{GQ} and \cite{CC} to \textit{asymptotically complex hyperbolic} (ACH) manifolds. Similar as in \cite{GQ} and \cite{CC}, the author here is particularly interested in (approximate) \textit{asymptotically complex hyperbolic Einstein} (ACHE) manifolds with infinity of positive CR-Yamabe type.

Before discussing the asymptotically complex hyperbolic manifolds, let us first recall some results for real asymptotically hyperbolic manifolds, which should help us to understand the complex case. Suppose $X$ is a manifold with boundary of dimension $n+1$ and  $\rho$ is a smooth boundary defining function. Let $g$ be a  smooth metric in the interior $\mathring{X}$ such that $\bar{g}=\rho^2g$ is smooth and nondegenerate up to the boundary, satisfying $|d\rho|^2_{\bar{g}}\rightarrow 1$ when $\rho\rightarrow 0$. Then in a collar neighbourhood of $M$, denoted by $[0,\epsilon)_{\rho}\times M$, 
$$
 \quad \bar{g}=d\rho^2+ g_{\rho}
$$
where $g_{\rho}$ a smooth-one parameter family of Riemannian metric on the boundary $M$. Then $g$ is asymptotically hyperbolic in the sense that all the sectional curvature has limit $-1$ when approaching to the boundary. The metric $g$ induces a conformal class $[g_0]$ on the boundary by choosing different boundary defining functions. 
A standard example is the ball model of real hyperbolic space $\mathbb{H}^{n+1}$, i.e. the unit ball $\mathbb{B}^{n+1}\subset \mathbb{R}_z^{n+1}$ equipped  with metric 
$$h=4(1-|z|^2)^{-2}dz^2.$$ 

The spectrum and resolvent for the Laplacian operator $\triangle_g$ for $(X, g)$ were studied by Mazzeo-Melrose \cite{MM}, Mazzeo \cite{Ma} and Guillarmou \cite{Gu1}. The spectrum of $\triangle_g$ consists of two disjoint parts,  the absolute continuous spectrum $\sigma_{ac}(\triangle_g)$ and the pure point spectrum $\sigma_{pp}(\triangle_g)$  (i.e. $L^2$-eigenvalues) . More explicitly, 
$$
\sigma_{ac}(\triangle_g)=\left[n^2/4, \infty\right), \quad \sigma_{pp}(\triangle_g)\subset \left(0,n^2/4\right). 
$$
The resolvent $R(\lambda)=(\triangle_g-\lambda(n-\lambda))^{-1}$ is a bounded operator on $L^2(X, \mathrm{dvol}_g)$ for $\lambda\in\mathbb{C}$,  $\mathrm{Re}(\lambda)>\frac{n}{2}$, $\lambda(n-\lambda)\notin\sigma_{pp}(\triangle_g)$, which has finite meromorphic extension to $\mathbb{C}\backslash \{\frac{n-1}{2}-K-\mathbb{N}_0\}$, where $2K$ is the order up to which the Taylor expansion of $g_{\rho}$ is even. For example, if  $g_{\rho}$ has complete even taylor expansion at $\rho=0$, then $R(\lambda)$ has finite meromorphic extension to entire $\mathbb{C}$. If $g$ is Einstein (i.e. $g$ is Poincar\'{e}-Einstein), then for $n$ odd,  the Taylor expansion of $g_{\rho}$ is even up to order $n-1$ while choosing $\rho$ to be the geodesic normal defining function. In this case, $R(\lambda)$ has meromorphic extension  to $\mathbb{C}\backslash (-\mathbb{N}_0)$. 
For $n$ even, a Poincar\'{e}-Einstein metric $g$ has logarithmic terms in the asymptotic expansion of $g_{\rho}$. This makes $\bar{g}=\rho^2g$ not smooth up to the boundary.  However,  the analysis of spectrum and resolvent above is still valid except that the mapping property of $R(\lambda)$ changes a bit. See \cite{CDLS} for the asymptotic expansion of Poincar\'{e}-Einstein metric and \cite{Gu2} for more details on meromorphic extension of the resolvent. In particular, Lee \cite{Le1} showed that for Poincar\'{e}-Einstein metric, if the conformal infinity is of nonnegative Yamabe type, then there is no $L^2$-eigenvalue and hence $\sigma(\triangle_g)=\sigma_{ac}(\triangle_g)=[n^2/4,\infty)$. 

The scattering operators associated to $\triangle_g$ are defined as follows: for any $f\in C^{\infty}(M)$ and $\mathrm{Re}(\lambda)>\frac{n}{2}$, $\lambda(n-\lambda)\notin\sigma_{pp}(\triangle_g)$, $2\lambda-n\notin\mathbb{N}$, there exists a unique solution to the following equation 
$$
\triangle_gu-\lambda(n-\lambda)u=0,
$$
such that
$$
u=x^{n-\lambda}F+x^{\lambda}G, \quad F, G\in C^{\infty}(X), \ F|_{M}=f.
$$
Then the scattering operator is defined by 
$$
S(\lambda)f=G|{M}. 
$$
Here $S(\lambda)$ is an elliptic pseudodifferential operator of order $2\lambda-n$, which is conformally covariant on the boundary. Moreover, $S(\lambda)$ can be extended meromorphically to $\mathbb{C}\backslash \{\frac{n-1}{2}-K-\mathbb{N}_0\}$, where $K$ is the same as above. The poles  at $\lambda_0>\frac{n}{2}$, $\lambda_0(n-\lambda_0)\in\sigma_{pp}(\triangle_g)$ or $2\lambda_0-n\in\mathbb{N}$, are of first order. For simplicity we define the renormalised scattering operators by
$$
P_{2\alpha}=2^{2\alpha}\frac{\Gamma(\alpha)}{\Gamma(-\alpha)}S\left(\frac{n}{2}+\alpha\right). 
$$
If  $g$ is \textit{approximate Einstein}, i.e.
$$
Ric_g+ng=\begin{cases}
\mathcal{O}(\rho^{\infty}) &\textrm{\ for $n$ odd},\\
\mathcal{O}(\rho^{n-2}) &\textrm{\ for $n$ even},
\end{cases}
$$
then $P_{2k}$ for $k=1,...,[\frac{n}{2}]$ are GJMS operators. In particular, $P_2$ is the Yamabe operator. See \cite{GZ} for more details. If the conformal infinity is of positive Yamabe type, the scatting operators are studied by Guillarmou and Qing \cite{GQ}. They showed that

\begin{theorem}[Guillarmou-Qing]
Let $(X,g)$ be a Poincar\'{e}-Einstein manifold of dimension $n+1>3$. The first scattering pole is less than $\frac{n}{2}-1$ if and only if its conformal infinity $(M,[g_0])$ is of positive Yamabe type. Moreover, in this case, for all $\alpha\in(0,1]$, $P_{2\alpha}$ satisfies 
\begin{itemize}
\item[(a)] the first eigenvalue is positive;
\item[(b)] $Q_{2\alpha}=P_{2\alpha}1$ is positive if choosing  $g_0$ with positive scalar curvature;
\item[(c)] the first eigenspace is generated by a single positive function;
\item[(d)] its Green function is nonnegative.
\end{itemize}
\end{theorem}

A new interpretation of the renormalised scattering operator $P_{2\alpha}$ is given by Case-Chang in \cite{CC} as a generalised Dirichlet-to-Neumann map on naturally associated smooth metric measure spaces. The authors exhibited some energy identities for $P_{2\alpha}$ on the boundary in terms of energies in the compact space $(X, \bar{g})$. This connects the positivity of renormalised scattering operators to the positivity of curvature terms  for the compactified metric $\bar{g}$. In particular, while a Poincar\'{e}-Einstein metric has positive conformal infinity, Qing \cite{Qi} showed that there exists a suitable compactification such that $\bar{g}$ has positive scalar metric, which implies that $P_{2\alpha}$ has positive spectrum for $\alpha\in(0,1)$ from Case-Chang's energy identity.

\vspace{0.2in}

In this paper, we consider  a complex manifold $X$ of complex dimension $n+1$,  with  strictly pseudoconvex boundary $M=\partial X$. The complex structure on $X$ naturally induces a CR-structure $(H,J)$ on M, where $H=\mathrm{Re} \mathcal{H}$,  $ \mathcal{H}=T_{1.0}X\cap \mathbb{C}TM$ and $J:H\longrightarrow H$ is defined by $J(V+\overline{V})=\sqrt{-1}(V-\overline{V})$.  Let $\rho$ be a smooth boundary defining function such that $\rho<0$ on $\mathring{X}$. Assume the function $-\log(-\rho)$ is plurisubharmonic. We consider the K\"{a}hler metric $g$ induced by K\"{a}hler form 
$$
\omega=-\frac{\sqrt{-1}}{2}\partial\overline{\partial} \log(-\rho).
$$
The metric  $g$ is asymptotically complex hyperbolic in the sense that the holomorphic sectional curvature has limit $-4$ when approaching to the boundary. More explicitly, $g$ takes the form
$$
g=\frac{h_{\rho}}{-\rho}+(1-r\rho)\left(\frac{(d\rho)^2}{4\rho^2}+\frac{(\theta_{\rho})^2}{\rho^2}\right)
$$
near $M$, where $h_{\rho}$ and $\theta_{\rho}$ have Taylor series in $\rho$ at $\rho=0$. In particular, $\theta_0$ gives the contact form on $M$ and $h_0$ induces a pseudo-Hermitian metric  on $H$. Moreover, the conformal class of the boundary pseudo-Hermitian structure is independent of choice of boundary defining function. A standard example is the complex hyperbolic space $\mathbb{H}_{\mathbb{C}}^{n+1}$: it  is the ball $\mathbb{B}^{n+1}=\{z\in \mathbb{C}^{n+1}: |z|<1\}$ equipped with K\"{a}hler metric $h^{\mathbb{C}}$ induced from the boundary defining function $\rho=|z|^2-1$.

The spectrum and resolvent of the Laplacian operator  $\triangle_g=\frac{1}{4}\delta d$ were studied by Epstein-Melrose-Mendoza \cite{EMM} and Vash-Wunsch \cite{VW}. Actually in both papers, the authors studied more general ACH manifolds.  Similar like the real case, the spectrum of $\triangle_g$ consists of two disjoint parts: the absolute continuous spectrum $\sigma_{ac}(\triangle_g)$ and the pure point spectrum $\sigma_{pp}(\triangle_g)$. More explicitly, 
$$
\sigma_{ac}(\triangle_g)=\left[(n+1)^2/4, \infty\right), \quad
\sigma_{pp}(\triangle_g)\subset \left(0, (n+1)^2/4\right). 
$$
The resolvent $R(s)=(\triangle_g-s(n+1-s))^{-1}$ is a bounded operator on $L^2(X,\mathrm{dvol}_g)$ for $s\in\mathbb{C}$, $\mathrm{Re}(s)>\frac{n+1}{2}$ and has finite meromorphic extension to $\mathbb{C}\backslash ( -\mathbb{N}_0)$. The smoothness of $\rho$ implies that the metric has even asymptotic expansion in the sense of \cite{EMM}. If $g$ is K\"{a}hler-Einstein, or equivalently if $\rho$ is a solution to the complex Monge-Amp\'{e}re equation,  then generally speaking $\rho$ is not smooth up to boundary and it has logarithmic terms in the taylor expansion at boundary with respect to smooth coordinates. See \cite{LM} for more details.  The analysis of spectrum and resolvent is still valid except that the mapping property of  $R(s)$ changes a bit. 
 
 The scattering operators associated to $\triangle_g$ is defined in a similar way as the real case. For any $f\in C^{\infty}(M)$ and $s\in\mathbb{C}$,  $\mathrm{Re}(s)>\frac{n+1}{2}$, $s(n+1-s)\notin\sigma_{pp}(\triangle_g)$ and $2s-(n+1)\notin \mathbb{N}$, consider the equation
 $$
 \triangle_gu -s(n+1-s)u=0.
 $$
There exists a unique solution $u$ such that
$$
u=\rho^{n+1-s}F+\rho^{s} G, \quad F, G\in C^{\infty}(X), \ F|_{M}=f.
$$
Then the scattering operator $S(s)$ is defined by 
$$S(s)f=G|_{M}, $$ 
which is a pseudo-differential operator of Heisenberg class of order $2(2s-n-1)$, conformally covariant and having meromorphic extension to $\mathbb{C}\backslash (-\mathbb{N}_0)$. See \cite{EMM} \cite{GS} for more details on the scattering theory and \cite{BG} \cite {EM} \cite{Ge}\cite{Po}\cite{Ta} for the Heisenberg calculus. For simplicity, we also define the renormalised scattering operators as 
$$
P_{2\gamma}=2^{2\gamma}\frac{\Gamma(\gamma)}{\Gamma(-\gamma)}S\left(\frac{n+1+\gamma}{2}\right). 
$$
If $g$ is approximate Einstein (see Definition \ref{def.app}), 
then $P_{2k}$ for $k=1,..,n+1$ are CR-GJMS operators of order $2k$. In particular, $P_2$ is the CR-Yamabe operator. See \cite{HPT}. This gives a different approach to construct the CR-invariant powers of sub-Laplacian studied by Grover-Graham \cite{GG} as well as the Q-curvature by Fefferman-Hirachi \cite{FH}.

The first result of this paper is parallel to the positivity result given in  \cite{GQ}: 
Here the definition of  \textit{approximate ACHE metric} is the same as in \cite{HPT}. 
See Definition \ref{def.app} for details. 
\begin{theorem}\label{thm.1}
Suppose $X$ is a complex manifold of complex dimension $n+1$ with strictly pseudoconvex boundary $M=\partial X$ and $g$ is an (approximate) ACHE metric of Bergman type given by K\"{a}hler form $\omega=-\frac{\sqrt{-1}}{2}\partial\overline{\partial} \log (-\rho)$ for some boundary defining function $\rho$. 
Assume $Ric_g\geq -2(n+2)g$ and the induced boundary CR-structure $(M,J,\theta_0)$ has positive Webster scalar curvature. 
Then for $\gamma\in(0,1)$, 
\begin{itemize}
\item[(a)] $Q_{2\gamma}=P_{2\gamma}1>0$;
\item[(b)] 
for $f\in C^{\infty}(M)$, $P_{2\gamma}f>0$ implies $f>0$ and $P_{2\gamma}f\geq 0$ implies $f\geq 0$;
\item[(c)] the bottom spectrum of $P_{2\gamma}$ is positive; more explicitly
$$
\oint_M fP_{2\gamma}f \theta\wedge(d\theta)^n \geq C_{\gamma}\oint_M |f|^2 \theta\wedge(d\theta)^n, 
$$
where $C_{\gamma}=\min_M Q_{2\gamma}>0$ on $M$. 
\end{itemize}
\end{theorem}

The second result  of this paper is parallel to the energy extension formulae given in \cite{CC}. In particular, while $(X, g)$ is the standard complex hyperbolic space $(\mathbb{H}^{n+1}_{\mathbb{C}}, h^{\mathbb{C}})$ and the CR-boundary is the Heisenberg group, Frank-Gonz\'{a}lez-Monticelli-Tan gave an answer to the extension problem in \cite{FGMT}. In the general case, the energy extension formulae is as follows: 

\begin{theorem}\label{thm.2}
Assume the same as in Theorem \ref{thm.1}. For $\gamma\in(0,1)$,  let $s=\frac{n+1+\gamma}{2}$ and $u$ solve the equation 
$\triangle_{g}u-s(n+1-s)u=0$ with Dirichlet data $f$. 
\begin{itemize}
\item[(a)]
Define $U$ by $u=(-\rho)^{n+1-s}U$. Then
$$
d_{\gamma}\oint_{M}fP_{2\gamma}f   \theta \wedge (d\theta)^n
=\int_X  (-\rho)^{-\gamma} \left\{\left[ (|d\rho|^2_{\rho}-\rho)\rho^{i\bar{j}}
-\rho^i\rho^{\bar{j}}\right] U_iU_{\bar{j}} + (n+1-s)^2U^2 \right\}\mathrm{dvol}_{\rho}\geq 0.
$$
\item[(b)] 
Define $W$ and $U'$ by $u=wU'=(\rho)^{n+1-s}WU'$ where $w$ solves the equation $\triangle_{\phi}w-s(n+1-s)w=0$ with Dirichlet data $1$.  Then
$$
d_{\gamma}\oint_{M}(fP_{2\gamma}f -Q_{2\gamma}f^2 ) \theta \wedge (d\theta)^n
=\int_X  (-\rho)^{-\gamma}W^2\left[ (|d\rho|^2_{\rho}-\rho)\rho^{i\bar{j}}
-\rho^i\rho^{\bar{j}}\right] U'_iU'_{\bar{j}} \mathrm{dvol}_{\rho}\geq 0. 
$$
\end{itemize}
Here 
$d_{\gamma}=-\gamma 2^{-2\gamma}\frac{\Gamma(-\gamma)}{\Gamma(\gamma)}>0$ and
$\mathrm{dvol}_{\rho}=
(\sqrt{-1})^{n+1}\det H(\rho) dz^1\wedge d\bar{z}^1\wedge \cdots \wedge dz^{n+1}\wedge d\bar{z}^{n+1}$. 
\end{theorem}

This paper is organised as follows. In Section 2, we introduce the geometric setting and previous results of the spectral and scattering theory in this background. In section 3, we talk about the case that the boundary CR-structure $(M,J,\theta)$ has positive (nonnegative) Webster scalar curvature and prove Part (a) and (b) of Theorem \ref{thm.1}. In Section 4, we prove the two energy extension formulae given in Theorem \ref{thm.2} and Part (c) of Theorem \ref{thm.1}.

\textbf{Acknowlegement}: The author would like to thank Paul Yang, Sun-Yung Alice Chang, Jefferey Case and Richard B. Melrose  for helpful conversation while this paper is prepared.

\vspace{0.2in}
\section{Geometric Setting}
Suppose $X$ is  a compact manifold of complex dimension $n+1$, with boundary $M=\partial X$. Denote by $\mathring{X}$ the interior of $X$. Let $\rho\in C^{\infty}(X)$ be a \textit{boundary defining function}, i.e. $\rho<0$ in $\mathring{X}$, $\rho=0$ on $M$ and 
$d\rho(p)\neq 0$ for all $p\in M$.  

We make two assumptions on $X$ and $\rho$ throughout the following paper. 
\begin{itemize}
\item[\textbf{(A1)}] The boundary $M$ is smooth and strictly pseudoconvex, i.e. there exists some boundary defining function $\tilde{\rho}\in C^{\infty}(X)$ such that $\tilde{\rho}$ is strictly plurisubharmonic near $M$. 
\item[\textbf{(A2)}] The function $\phi=-\log(-\rho)$ is  strictly plurisubharmonic all over $\mathring{X}$.
\end{itemize}
%

\subsection{ACH K\"{a}hler manifold }\label{sec.ach}

The \textit{K\"{a}hler form} associated to $\phi=-\log(-\rho)$ is defined by
\begin{equation}\label{kahlerform}
\omega_{\phi} =\frac{\sqrt{-1}}{2}  \partial \overline{\partial} \phi=-\frac{\sqrt{-1}}{2} \partial \overline{\partial} \log(-\rho)
=\frac{\sqrt{-1}}{2} \left(\frac{\partial \overline{\partial} \rho}{-\rho}+\frac{\partial \rho\wedge \overline{\partial}\rho}{\rho^2}\right). 
\end{equation}
By assumption (A2), $\omega_{\phi}$ is a positive real $(1,1)$-form and hence induces a Riemannian (K\"{a}hler) metric $g_{\phi}$:
\begin{equation}\label{kahlermetric}
g_{\phi}(V,W)=\omega_{\phi}(V,JW), \quad \forall\ V, W\in T\mathring{X};
\end{equation}
and a Hermitian metric $h_{\phi}$: 
\begin{equation}\label{hermitianmetric}
h_{\phi}(V,\overline{W})= -2\sqrt{-1}\omega_{\phi}(V, \overline{W}),  \quad \forall\  V, W\in T_{1,0}\mathring{X}. 
\end{equation}
%
In local holomorphic coordinates $\{z^{i}: i=1,...,n+1\}$,
\begin{equation}
\begin{aligned}
\omega_{\phi}=\frac{\sqrt{-1}}{2}\phi_{i\bar{j}} dz^i\wedge d\bar{z}^j,
\quad g_{\phi}=\frac{1}{2}\phi_{i\bar{j}} dz^i\odot d\bar{z}^j,
\quad h_{\phi} =\phi_{i\bar{j}} dz^i\otimes d\bar{z}^j. 
\end{aligned}
\end{equation}
Notice that in our convention the K\"{a}hler metric has components $g_{i\bar{j}}=\frac{1}{2}\phi_{i\bar{j}}$.  Denote by $H(\phi)=[\phi_{i\bar{j}}]$ and $H(\rho)=[\rho_{i\bar{j}}]$ the Hessians of $\phi$ and $\rho$. Then
\begin{equation*}
\phi_{i\bar{j}}=\frac{\rho_{i\bar{j}}}{-\rho}+\frac{\rho_{i}\rho_{\bar{j}}}{\rho^2}, \quad
\rho_{i\bar{j}}=(-\rho)\left(\phi_{i\bar{j}}-\phi_i\phi_{\bar{j}}\right). 
\end{equation*}
Assumption (A2) says  $H(\phi)$ is positive definite in $\mathring{X}$ and hence invertible. Let $[\phi^{i\bar{j}}]$ be the inverse of $H(\phi)$ and denote
\begin{equation*}
\phi^i=\phi_{\bar{j}}\phi^{i\bar{j}}, \quad 
\phi^{\bar{j}}=\phi_{i}\phi^{i\bar{j}}, \quad 
|d\phi|_{\phi}^2=\phi_i\phi_{\bar{j}}\phi^{i\bar{j}}. 
\end{equation*}
Then
\begin{equation}\label{eq.2.5}
\det H(\rho)=(-\rho)^{n+1}\left(1-|d\phi|^2_{\phi}\right) \det H(\phi). 
\end{equation}
\begin{lemma}
Given $H(\phi)$ positive definite, then $H(\rho)$ is positive definite if and only if 
$$|d\phi|^2_{\phi}<1.$$
\end{lemma}

If $H(\rho)$ is also positive definite with inverse $[\rho^{i\bar{j}}]$, then
\begin{equation}\label{eq.2.6}
\rho^{i\bar{j}}=\frac{1}{-\rho}\left(\phi^{i\bar{j}}-\frac{\phi^i\phi^{\bar{j}}}{|d\phi|^2_{\phi}-1}\right). 
\end{equation}
In this case we can denote
\begin{equation*}
\rho^i=\rho_{\bar{j}}\rho^{i\bar{j}}, \quad 
\rho^{\bar{j}}=\rho_{i}\rho^{i\bar{j}}, \quad 
|d\rho|_{\rho}^2=\rho_i\rho_{\bar{j}}\rho^{i\bar{j}}.
\end{equation*}
Then
\begin{equation}\label{eq.2.2}
\det H(\phi)=(-\rho)^{-(n+2)} (|d\rho|^2_{\rho}-\rho)\det H(\rho). 
\end{equation}
Moreover we can recover $[\phi^{i\bar{j}}]$ from $[\rho^{i\bar{j}}]$ by 
\begin{equation}\label{eq.2.3}
\phi^{i\bar{j}}=(-\rho)\left(\rho^{i\bar{j}}-\frac{\rho^i\rho^{\bar{j}}}{|d\rho|^2_{\rho}-\rho}
\right). 
\end{equation}
%
%
%
%
%

A direct computation by Li-Wang \cite{LW} shows that all the curvature tensor is approximately constant when approaching to the boundary. More explicitly
\begin{equation*}
R_{i\bar{j}k\bar{l}} =\frac{1}{2}\left(\phi_{i\bar{j}} \phi_{k\bar{l}}+ \phi_{k\bar{j}} \phi_{i\bar{l}}\right)+\mathcal{O}\left(\frac{1}{-\rho}\right)
=2\left(g_{i\bar{j}} g_{k\bar{l}}+ g_{k\bar{j}} g_{i\bar{l}}\right)+\mathcal{O}\left(\frac{1}{-\rho}\right). 
\end{equation*}
Hence the K\"{a}hler metric $g_{\phi}$ is \textit{asymptotically complex hyperbolic} (ACH) with all the holomorphic sectional curvature having limit $-4$ when $\rho\rightarrow 0$.
The Ricci curvature tensor for the K\"{a}hler  metric $g_{\phi}$ can be expressed by
\begin{equation}\label{eq.2.7}
R_{i\bar{j}}= -\left( \log \det H(\phi)\right)_{i\bar{j}} =-(n+2)\phi_{i\bar{j}}- \left( \log J[\rho]\right)_{i\bar{j}}, 
\end{equation}
where $J[\rho]$ is the Fefferman operator defined by
\begin{equation}\label{eq.2.8}
J[\rho]=-\det \left[\begin{array}{cc}
\rho & \rho_{\bar{j}}\\ \rho_i &\rho_{i\bar{j}}
\end{array}\right].
\end{equation}
in local holomorphic coordinates. Using above notation, 
\begin{equation}\label{eq.2.9}
J[\rho]=(|d\rho|^2_{\rho}-\rho)\det H(\rho)=e^{-(n+2)\phi}\det H(\phi). 
\end{equation}
%
%
%
If $J[\rho]=1$,  it is easy to see from (\ref{eq.2.7}) that the K\"{a}hler metric $g_{\phi}$ is Einstein, i.e.
$$
Ric_{\phi}=-2(n+2)g_{\phi}.
$$
We call such $g_{\phi}$ \textit{asymptotically complex hyperbolic Einstein} (ACHE) metric  of Bergman type and will discuss this more in Section \ref{sec.ake}.


\subsection{CR boundary.} The complex structure on $X$ defined in Section \ref{sec.ach} induces a CR-structure on the boundary $M$. For the purpose of studying the asymptotic behaviour of metric $g_{\phi}$ defined by (\ref{kahlermetric}),  we will also discuss the CR structure on each level set $M_{\epsilon}=\{\rho=-\epsilon\}$ for sufficiently small $\epsilon> 0$. Notice that there exists  some $\epsilon_0>0$ such that for each $0\leq \epsilon<\epsilon_0$, $M_{\epsilon}$ is a smooth manifold of real dimension $2n+1$. Denote by $U=\{-\epsilon_0<\rho\leq 0 \}$  the collar neighbourhood of $M$ in $X$. 

The \textit{CR-structure} $(H_{\epsilon}, J)$ on $M_{\epsilon}$ is defined by the hyperplane bundle $H_{\epsilon}=\mathrm{Re}\mathcal{H}_{\epsilon}$ with
\begin{equation*}
\mathcal{H}_{\epsilon}=\mathbb{C}TM_{\epsilon}\cap T_{1,0}U, 
\end{equation*}
and a smooth bundle map
\begin{equation*}
\begin{aligned}
J: H_{\epsilon}&\longrightarrow H_{\epsilon}
\\
(V+\overline{V}) &\longrightarrow \sqrt{-1}(V-\overline{V}), 
\end{aligned}
\end{equation*}
which gives the almost complex structure on $H_{\epsilon}$. 
From the definition, the CR-structure is automatically integrable since $[\mathcal{H}_{\epsilon}, \mathcal{H}_{\epsilon}]\subset \mathcal{H}_{\epsilon}$. Associated to the boundary defining function $\rho$, there is a real one-form
\begin{equation*}
\Theta=\frac{\sqrt{-1}}{2}(\overline{\partial}\rho-\partial\rho),
\end{equation*}
which induces the contact form $\theta_{\epsilon}=i^*_{\epsilon}\Theta$ on each $M_{\epsilon}$ by the pullback map of embedding $i_{\epsilon}: M_{\epsilon}\rightarrow X$. 
The Levi form on $M_{\epsilon}$ is defined by 
\begin{equation*}
L_{\theta_{\epsilon}}=-\sqrt{-1}d\theta_{\epsilon}.
\end{equation*}
For $V,W\in \mathcal{H}_{\epsilon}$, 
\begin{equation*}
L_{\theta_{\epsilon}}(V, \overline{W})=-\sqrt{-1}d\theta_{\epsilon}(V, \overline{W})=\partial\overline{\partial}\rho(V,\overline{W}). 
\end{equation*}
A continuity argument shows that
$L_{\theta_{\epsilon}}$ is positive definite for sufficiently small $\epsilon\geq 0$ since $M$ is strictly pseudoconvex. Hence it defines the pseudo-Hermitian structure on each
 $M_{\epsilon}$ for $\epsilon\geq 0$ small. 
 
Notice that if $\tilde{\rho}=e^F \rho$ for some $F\in C^{\infty}(X)$, then $\tilde{\rho}$ also induces a contact form $\tilde{\theta}_{\epsilon}$ on 
$\tilde{M}_{\epsilon}=\{\tilde{\rho}=-\epsilon\}$, as well as  the Levi form $L_{\tilde{\theta}_{\epsilon}}$. By simple computation, 
\begin{equation*}
\tilde{\theta}_0=e^f\theta_0,\quad L_{\tilde{\theta}_0}=e^fL_{\theta_0}, 
\end{equation*}
where $f=F|_{M}$. 
Therefore the conformal class of the contact form and pseudo-Hermitian structure on $M$ is independent of choice of boundary defining function. Generally speaking, this is not true on $M_{\epsilon}$ if $\epsilon>0$. 

Let $\Xi$ be a  (1,0)-vector filed on the collar neighborhood $U$ of $M$ and satisfy
\begin{equation*}
\partial\rho(\Xi)=1, \quad \Xi \righthalfcup\partial\overline{\partial}\rho=r\overline{\partial}\rho, 
\end{equation*}
for some $r\in C^{\infty}(U)$. Such $\Xi$ is uniquely determined by $\rho$. And $r$ is called the \textit{transverse curvature}. Decompose $\Xi$ into real part and imaginary part: 
\begin{equation*}
\Xi=\frac{1}{2}(N-\sqrt{-1}T),
\end{equation*}
where $N, T$ are real vector fields on $U$. It is easy to show that
\begin{equation*}
d\rho(N)=2, \quad d\rho(T)=0, \quad \Theta(N)=0,\quad
\Theta(T)=1, \quad T\righthalfcup d\theta_{\epsilon}=0.
\end{equation*}
Hence $T$ is the \textit{characteristic vector field} for each $(M_{\epsilon}, J, \theta_{\epsilon})$ and $N$ is normal to $M_{\epsilon}$.

Next we write some explicit formulae in local frame that is compatible with the CR-structure on $M_{\epsilon}$. For simplicity, denote
$$
\theta=\theta_{\epsilon}, \quad
\mathcal{H}=\mathcal{H}_{\epsilon}, \quad
L_{\theta} = L_{\theta_{\epsilon}}. 
$$   
Let $\{W_{\alpha}: \alpha=1,...,n\} $ be a local frame for $\mathcal{H}$. Then $\{W_{\alpha}, W_{\overline{\alpha}}, T: \alpha=1,...,n\}$ forms a local frame for $\mathbb{C}TM_{\epsilon}$ and $\{W_{\alpha}, W_{\overline{\alpha}}, T, N: \alpha=1,...,n\}$ forms a local frame for $TU$. Let $\theta^{\alpha}$ be the dual form of $W_{\alpha}$, then $\{\theta^{\alpha}, \theta^{\overline{\alpha}},\theta: \alpha=1,...,n\}$ is a dual coframe for $\mathbb{C}T^*M_{\epsilon}$ and $\{\theta^{\alpha}, \theta^{\overline{\alpha}},\partial\rho,\overline{\partial}\rho: \alpha=1,...,n\}$ is a dual coframe for $T^*U$. 

The Levi form on each $\mathcal{H}$ is given by
\begin{equation*}
L_{\theta}=h_{\alpha\bar{\beta}}\theta^{\alpha}\wedge \theta^{\bar{\beta}}, 
\end{equation*}
where $h_{\alpha\bar{\beta}}$ is a smooth function on $U$ valued in Hermitian matrix. 
Then
\begin{equation*}
\partial\overline{\partial}\rho =h_{\alpha\bar{\beta}}\theta^{\alpha}\wedge \theta^{\bar{\beta}}
+r\partial \rho\wedge \overline{\partial}\rho. 
\end{equation*}
It is easy to see that $r$ is a real function. Near $M$, $\rho$ is strictly plurisubharmonic if and only if $r>0$. 
Then the K\"{a}hler form $\omega_{\phi}$ can be expressed by
\begin{equation}\label{kahlerform2}
\omega_{\phi}= \sqrt{-1}\left(\frac{1}{-\rho} h_{\alpha\bar{\beta}}\theta^{\alpha}\wedge \theta^{\bar{\beta}}+
\frac{1-r\rho}{\rho^2}\partial \rho\wedge \overline{\partial}\rho\right), 
\end{equation}
and the induced  metric $g_{\phi}$ is given by
\begin{equation}\label{kahlermetric2}
g_{\phi} =\frac{1}{2}\left(\frac{1}{-\rho} h_{\alpha\bar{\beta}}\theta^{\alpha}\varodot \theta^{\bar{\beta}}+
\frac{1-r\rho}{\rho^2}\partial \rho \varodot \overline{\partial}\rho\right). 
\end{equation}
Denote
\begin{equation*}
W_{n+1}=\Xi, \quad W_{\overline{n+1}}=\overline{\Xi}, \quad
\theta^{n+1}=\partial_{\rho}, \quad \theta^{\overline{n+1}}=\overline{\partial}\rho. 
\end{equation*}
In what follows, we use Greek indices $\alpha, \beta, \cdots $ as integers chosen  from $\{1,...,n\}$ and Latin indices $i, j, \dots$ as integers chosen from $\{1,...,n+1\}$. 

Graham and Lee proved in \cite{GL}:
\begin{proposition}
There exists a unique linear connection $\nabla$ on $U$ such that 
\begin{itemize}
\item[(a)] For any vector field $V, W$ tangent to some $M_{\epsilon}$, $\nabla_VW=\hat{\nabla}_V^{\epsilon}W$, where $\hat{\nabla}^{\epsilon}$ is the pseudo-Hermitian connection on $M_{\epsilon}$. 
\item[(b)] $\nabla$ preserves $\mathcal{H}$, $N$, $T$ and $L_{\theta}$: for any $X\in TU$,  $\nabla_X\mathcal{H}\subset \mathcal{H}$ and $\nabla T=\nabla N=\nabla L_{\theta}=0$. 
\item[(c)] If $\{W_{\alpha}: \alpha=1,...,n\}$ is a frame of $\mathcal{H}$ and $\{\theta^{\alpha}, \partial_{\rho}: \alpha=1,...,n\}$ is the dual $(1,0)$-coframe on $U$, then the connection $1$-forms, defined by $\nabla W_{\alpha}=\rho_{\alpha}^{\ \beta}\varotimes W_{\beta}$ satisfy the following structure equation:
\begin{equation*}
d\theta^{\alpha}=\theta^{\beta}\wedge \rho_{\beta}^{\ \alpha} -\sqrt{-1} \partial\rho \wedge \tau^{\alpha} +\sqrt{-1}(W^{\alpha}r) d\rho \wedge \theta +\frac{1}{2} r d\rho\wedge \theta^{\alpha}. 
\end{equation*}
\end{itemize}
Here $\nabla$ is called the ambient connection defined by $\rho$.
\end{proposition}

The ambient connection was used to study the asymptotic behaviour of Laplacian operator $\triangle_{\phi}$ in \cite{GL} and \cite{HPT}. 


\subsection{Approximate ACHE metric.} \label{sec.ake}
An \textit{approximate} asymptotically complex hyperbolic Einstein (ACHE) metric of Bergman type  defined on a complex domain is given by an \textit{approximate} solution to the Monge-Amp\`{e}re equation. 

Let  $\Omega$ be a pseudoconvex domain in $\mathbb{C}^{n+1} $ with smooth boundary $\partial \Omega =\Sigma$. If $\rho\in C^{\infty}(\mathring{\Omega})$ solves  the Monge-Amp\`{e}re equation 
\begin{equation}\label{mongeampere}
\begin{cases}
J[\rho]=1&\ \textrm{in $\mathring{\Omega}$}, \\
\rho|_{\Omega}=0, &\ \textrm{on $\Sigma$},
\end{cases}
\end{equation}
then the K\"{a}hler metric $g_{\phi}$,  induced by K\"{a}hler form $\omega_{\phi}=\frac{\sqrt{-1}}{2}\partial\partial\phi$,  with $\phi=- \log (-\rho)$, is K\"{a}hler-Einstein. 
Feffermen \cite{Fe} showed that for any strictly pseudoconvex domain $\Omega\subset\mathbb{C}^{n+1}$ with smooth boundary $\Sigma$,  there exists an approximate solution $\rho\in C^{\infty}(\Omega)$ satisfies
\begin{equation*}
\begin{cases}
J[\rho]=1+\mathcal{O}(\rho^{n+2})&\ \textrm{in $\mathring{\Omega}$}, \\
\rho|_{\Omega}=0, &\ \textrm{on $\Sigma$}. 
\end{cases}
\end{equation*}
Cheng and Yau \cite{CY} showed the existence and uniqueness of the exact solution to (\ref{mongeampere}) which is in $C^{\infty}(\mathring{\Omega})\cap C^{n+5/2-\epsilon}(\Omega)$.  Lee and Melrose \cite{LM} investigated the asymptotic behaviour of this exact solution and showed that the solution has only conormal singularity expressed by logarithmic terms begins at order $n+3$.  

Hislop, Perry and Tang \cite{HPT} extended Fefferman's approximate solution to complex manifolds with strictly peudoconvex boundary. 
\begin{definition}\label{def.app}
We call a K\"{a}hler metric $g$ on $\mathring{X}$ an approximate ACHE metric of Bergman type if  $g$ is induced by the K\"{a}hler form $\omega=-\frac{\sqrt{-1}}{2}\partial \overline{\partial}\log(-\rho)$ where $\rho$ is a global approximate  solution to the Monge-Amp\`{e}re equation in the following sense:   $\rho$ is a smooth boundary defining function and for any $p\in M$, there exists a neighbourhood  $U\subset X$ and a set of holomorphic coordinates  $\{z^i: i=1,...,n+1\}$ such that
\begin{equation*}
\begin{cases}
J[\rho]=1+\mathcal{O}(\rho^{n+2})&\ \textrm{in $U\cap\mathring{X}$}, \\
\rho=0, &\ \textrm{on $U\cap M$}. 
\end{cases}
\end{equation*}
\end{definition}

%
%
The condition of existence of global approximate  Monge-Amp\`{e}re solution is given by 
Hislop, Perry and Tang \cite{HPT}:

\begin{proposition}
Suppose $X^{n+1}$ is a compact complex manifold with boundary $M=\partial X$. There exists a global approximate solution $\rho$ to the Monge-Amp\`{e}re equation in a neighbourhood of $M$ if $M$ admits a pseudo-Hermitian structure $\theta$ with the following property: in a neighbourhood of any point $p\in M$, there is a local closed $(n+1,0)$ form $\xi$ such that $\theta$ is volume-normalised with respect to $\xi$.  
\end{proposition}

For $n\geq 2$, the condition of existence of global approximate solution to the Monge-Amp\`{e}re equation  can be given in a more geometric way, followed from the result given by Lee \cite{Le2}: 

\begin{proposition}
For $\mathrm{dim}M\geq 5$, a contact form $\theta$ on $M$ is pseudo-Einstein if and only if for each $p\in M$ there is a neighbourhood of $p$ in $M$ and a locally defined closed section $\xi$ of the canonical bundle with respect to which $\theta$ is volume-normalised. 
\end{proposition}

\begin{remark}
Given a CR-manifold, a pseudo-Hermitian structure is called \textit{pseudo-Einstein} if the pseudo-Hermitian Ricci tensor is a scalar multiple of the Levi form.
This concept is very different from its analogue in Riemannian geometry. For example, 
the pseudo-Einstein condition does not imply that the scalar curvature is a constant, due to the presence of torsion in the Bianchi identity in pseudo-Hermitian geometry. 
\end{remark}


\subsection{Beltrami-Laplacian opeator.}
In this paper, the Beltrami-Laplacian operator associated to the metric $g_{\phi}$ is defined by $\triangle_{\phi}=\frac{1}{4}\delta d$. For a real function $u$
\begin{equation}
\triangle_{\phi}u=-\phi^{i\bar{j}}u_{i\bar{j}}
\end{equation}
in local holomorphic coordinates $\{z^i: i=1,...,n+1\}$. 
In \cite{GL} Graham and Lee used the ambient connection to write $\triangle_{\phi}$ in the local frame $\{W_{\alpha}, W_{\bar{\alpha}}, \Xi, \overline{\Xi}: \alpha=1,..., n\}$ which is  compatible with the boundary CR-structure. They showed that near boundary, 
\begin{equation}\label{laplaceasymp.1}
\triangle_{\phi}=\frac{{\rho}}{4}\left[ \frac{-\rho}{1-r\rho} (N^2+T^2+2rN+2X_{r})-2\triangle_b +2n N
\right]
\end{equation}
where 
\begin{equation*}
X_r=r^{\alpha} W_{\alpha}+r^{\bar{\alpha}} W_{\bar{\alpha}}
\end{equation*}
and $\triangle_b$ is the sub-Laplacian defined on each $M_{\epsilon}$ by
\begin{equation*}
\triangle_bu=-(u_{\alpha}^{\ \alpha}+u_{\bar{\beta}}^{\ \bar{\beta}})
\end{equation*}
where the covariant derivatives are taken w.r.t. the Tanaka-Webster connection on each $M_{\epsilon}$.  Let $x=-\rho$. Then
\begin{equation*}
N=-2\partial_x, 
\end{equation*}
and 
\begin{equation}
\triangle_{\phi}=-\left(\frac{1}{1+rx}\right)(x\partial_x)^2 +(n+1)x\partial_x+\frac{x}{2}\triangle_b
 -\frac{1}{4}\left(\frac{x^2}{1+rx}\right)(T^2+2X_r). 
\end{equation}
Write $\triangle_{\phi}=\sum_{k=0}^{\infty}x^kL_k$. Then
\begin{equation}\label{laplaceasymp.2}
L_0 =-(x\partial_x)^2+(n+1)x\partial_x, \quad 
L_1 =\frac{1}{2}\triangle_b+r_0(x\partial_x)^2, 
\end{equation}
where $r=r_0+O(x)$. This is computed in \cite{HPT} by Hislop, Perry  and Tang. 

Epstein, Melrose and Mendoza \cite{EMM} studied the spectrum and resolvent for Laplacian of more general ACH manifolds. In our paper, the metric is of Bergman type and hence is even in the sense of \cite{EMM} and \cite{GS}. 

\begin{proposition}
The spectrum of $\triangle_{\phi}$ consists of two parts: absolute continuous spectrum $\sigma_{ac}(\triangle_{\phi})$ and pure point spectrum $\sigma_{pp}(\triangle_{\phi})$, which corresponds to the $L^2$-eigenvalues. Moreover, 
$$
\sigma_{ac}(\triangle_{\phi})=\left[(n+1)^2/4, \infty\right), \quad
\sigma_{pp}(\triangle_{\phi})\subset \left(0,(n+1)^2/4\right). 
$$
The resolvent $R(s)=(\triangle_{\phi}- s(n+1-s))^{-1}$ is a bounded operator on $L^2$ for $s\in \mathbb{C}$, $\mathrm{Re} s>\frac{n+1}{2}$ and has a finite meromorphic extension to $\mathbb{C}\backslash (-\mathbb{N}_0)$ as a map: $\dot{C}^{\infty}(X)\rightarrow C^{\infty}(X)$. 
\end{proposition}


\subsection{Scattering operators. }
For $f\in C^{\infty}(M)$ and $s\in \mathbb{C}$,  $\mathrm{Re} s>\frac{n+1}{2}$, $s(n+1-s)\notin \sigma_{pp}(\triangle_{\phi})$ and $2s-(n+1)\notin \mathbb{N}$, there exists a unique solution $u$ satisfying the equation
$$
\triangle_{\phi} u- s(n+1-s)u=0, 
$$
and having asymptotic expansion at boundary in the following form:
$$
u=(-\rho)^{n+1-s}F+(-\rho)^{s} G, \quad
F, G\in C^{\infty} (X),\ F|_{M}=f. 
$$
Hence we can define the scattering operator to be:
$$
\begin{aligned}
S(s): C^{\infty}(M)&\longrightarrow C^{\infty}(M) \\
f&\longrightarrow G|_{M}. 
\end{aligned}
$$
Guillarmou-S\'{a} Barreto \cite{GS} studied the scattering operators for more general ACH manifolds. In our setting,
\begin{proposition}
The scattering operators $S(s)$ is a family of conformally covariant pseudo-differential operators of Heisenberg class $\Psi_{\theta}^{2(2s-n-1)}(M)$
associated to the contact form $\theta=\Theta|_{M}$,
where $
\theta=\frac{\sqrt{-1}}{2}(\bar{\partial}\rho-\partial \rho)$,
with principal symbol
$$
\sigma_{pr}(S(s))=c_n\frac{2^{2s+1}\Gamma(s)^2}{\Gamma(2s-n-1)}\mathcal{F}_{V\rightarrow\xi}(\|V\|^{-4s}_{He}). 
$$
Here $\|V\|_{He}$ denotes the Heisenberg norm on $TM$
$$
\|V\|_{He}=\left(4\theta(V)^2+\frac{1}{2}d\theta(V,JV)^2\right)^{\frac{1}{4}}
$$
and $\mathcal{F}_{V\rightarrow\xi}$ denote the Fourier Transform from $TM$ to $T^*M$. 

Moreover $S(s)$ extends meromorphically to $\mathbb{C}\backslash (-\mathbb{N}_0)$.
It has at most poles of order $1$ at each $s_k=\frac{n+1+k}{2}$ with $k\in \mathbb{N}$, the residue of which is a Heisenberg differential operator in $\Psi^k_{\theta}(M)$ plus a projector appearing if and only if $s_k(n+1-s_k)\in \sigma_{pp}(\triangle_{\phi})$. At $s_{k}$, we have,
$$
\mathrm{Res}_{s_{k}} S(s) =\frac{1}{2^{2k}((k-1)!k!)} \Pi_{l=1}^k (-\triangle_b+\sqrt{-1}(k+1-2l)T), \quad \mathrm{mod}\  \Psi_{\theta}^{2k-1}(M). 
$$
\end{proposition}
We recommend the reader to seek for more details on the Heisenberg class of pseudo-differential operators in \cite{BG} by Beals-Greiner, \cite{Ta} by Taylor and \cite{Po} by Ponge. 

For  approximate ACHE manifolds, Hislop-Perry-Tang showed that the residues at certain poles are the CR-covariant GJMS operators on the CR-boundary:
\begin{proposition}
If $[(n+1)^2-k^2]/4\notin \sigma_{pp}(\triangle_{\phi})$, the scattering operator $S(s)$, associated to an approximate ACHE metric  $g_{\phi}$ of Bergman type has single poles at $(n+1+k)/2$ with residue
$$
\mathrm{Res}_{s=\frac{n+1+k}{2}} S(s) = c_k P_{2k}
$$
where $P_{2k}$ are CR-covariant  differential operators of order $2k$ for $1\leq k\leq n+1$ and
$$
c_k=\frac{(-1)^k}{2^{2k}k!(k-1)!}. 
$$
\end{proposition}
For simplicity, we define the renormalised scattering operators by
\begin{equation}\label{rescatteringop}
P_{2\gamma}=2^{2\gamma}\frac{\Gamma(\gamma)}{\Gamma(-\gamma)} S\left(\frac{n+1+\gamma}{2}\right)
\end{equation}
and the fractional Q-curvature by
\begin{equation*}
Q_{2\gamma}=P_{2\gamma}1. 
\end{equation*}
In particular, when $\gamma=1$, 
\begin{equation*}
P_2=\triangle_b+\frac{n}{2(n+1)} R_{\theta}, \quad Q_2=\frac{n}{2(n+1)}R_{\theta}
\end{equation*}
where $R_{\theta}$ is the Webster scalar curvature and $P_2$ is the CR-Yamabe operator of Jerison and Lee \cite{JL}.


\vspace{0.2in}
\section{Nonegative CR-Boundary}
Throughout this section, we assume (A1)-(A2) and
\begin{itemize}
\item[\textbf{(A3)}] The boundary defining function $\rho$  is a global (approximate) solution to the Monge-Amp\`{e}re equation (see Definition \ref{def.app});
\item[\textbf{(A4)}] $Ric_{\phi}\geq -2(n+2)g_{\phi}$ in $\mathring{X}$,  where $g_{\phi}$ is the induced K\"{a}hler metric by  the K\"{a}hler form $\omega_{\phi}=\frac{\sqrt{-1}}{2}\partial \overline{\partial}\phi$ with $\phi=-\log (-\rho)$. 
\end{itemize}
Then  $g_{\phi}$  is (approximate) ACHE and the one-form $\Theta=\frac{\sqrt{-1}}{2}(\overline{\partial}\rho-\partial \rho)$ induces a pseudo-Hermitian structure which is pseudo-Einstein on the boundary $M$. 

\subsection{Preliminary Computation}
First we show that (approximate) ACHE solution forces a compatibility relation between the Webster scalar curvature and the transverse curvature at boundary. More explicitly, 
\begin{lemma} \label{lem.3.1} Assumption (A3) implies that the Webster scalar curvature $R_{\theta}$ and the transverse curvature $r$ satisfy
\begin{equation*}
r= \frac{1}{n(n+1)}R_{\theta}+\mathcal{O}(\rho). 
\end{equation*}
\end{lemma}
\begin{proof}
Since $\rho$ is a global (approximate) Monge-Amp\`{e}re solution, then for any $p$ in $M$, there exists a local holomorphic chart $\{z^i: i=1,...,n+1\}$ such that
\begin{equation*}
J[\rho]=1+\mathcal{O}(\rho^{n+2}). 
\end{equation*}
According to the formulae given by Li and Luk \cite{LL}, the Webster pseudo-Ricci curvature on $M$ is given by 
\begin{equation*}
\mathrm{Ric}_{\theta}\left(W, \overline{V}\right)
=-\left(\partial \overline{\partial}\log J[\rho]\right)(W, \overline{V}) 
+(n+1)\frac{\det H(\rho)}{J[\rho]}  L_{\theta}\left(W, \overline{V}\right). 
\end{equation*}
Hence
\begin{equation*}
R_{\theta}=n(n+1)\frac{\det H(\rho)}{J[\rho]}\vline_M. 
\end{equation*}
Recall that $\det H(\rho)=(-\rho)^{n+1}(1-|d\phi|^2_{\phi})\det H(\phi)$ and $J[\rho]=e^{-\phi}(-\rho)^{n+1}\det H(\phi)$. It is easy to see that near boundary
\begin{equation*}
\frac{\det H(\rho)}{J[\rho]} 
=e^{\phi}(1-|d\phi|^2_{\phi})=\frac{r}{1-r\rho} = r+\mathcal{O}(\rho). 
\end{equation*}
\end{proof}

Notice that (A3) also implies that at $\rho=0$ the boundary pseudo-Hermitian structure is pseudo-Einstein. 

\begin{lemma}\label{lem.3.2}
Assume (A1)-(A4). If the Webster scalar curvature $R_{\theta}\geq 0$ on $M$ then in $\mathring{X}$, 
$$
I=e^{\phi}\left(1-|d\phi|^2_{\phi}\right)>0. 
$$
\end{lemma}
\begin{proof}
First we recall some computation by Li-Wang in \cite{LW} to prove the case of $R_{\theta}>0$. In local holomorphic coordinates $\{z^i: i=1,...,n+1\}$, using $-\triangle_{\phi}\phi=\phi^{i\bar{j}}\phi_{i\bar{j}}=n+1$, we have 
$$
-\triangle_{\phi}(|d\phi|^2_{\phi})=
\phi^{i\bar{j}}(|d\phi|^2_{\phi})_{i\bar{j}} = R_{k\bar{j}}\phi^{k}\phi^{\bar{j}} +n+1+\phi_{k;i}\phi^{k;i}. 
$$
Here we use "$;$" to denote the covariant derivative w.r.t. metric $g_{\phi}$ and
$$
\phi^{k;i}=\phi^{k\bar{l}}\phi^{i\bar{j}}\phi_{\bar{l};\bar{j}}, \quad \phi^k=\phi^{k{\bar{l}}}\phi_{\bar{l}},
\quad \phi^{\bar{j}}=\phi^{i\bar{j}}\phi_{i}. 
$$ 
Then by Cauchy-Schwarz inequality and assumption (A4), we have
$$
\begin{aligned}
\triangle_{\phi}I=&\ 
-\phi^{i\bar{j}}\left\{e^{\phi}(\phi_i\phi_{\bar{j}}+\phi_{i\bar{j}})(1-|d\phi|^2_{\phi}) - e^{\phi}[ (|d\phi|^2_{\phi})_i\phi_{\bar{j}}+(|d\phi|^2_{\phi})_{\bar{j}}\phi_i] -e^{\phi}(|d\phi|^2_{\phi})_{i\bar{j}}
\right\}
\\
\geq &\ e^{\phi}\left[ |d\phi|^4_{\phi} +\phi_{k;i}\phi^{k}\phi^i+\phi^{k;i}\phi_k\phi_i +\phi_{k;i}\phi^{k;i}
\right]
\\
\geq&\  0. 
\end{aligned}
$$
From the proof of Lemma \ref{lem.3.1}, 
$$
I|_{M}=r|_{M}=\frac{1}{n(n+1)}R_{\theta}\geq 0. 
$$ 
Applying  the strong maximum principle and we get $I\geq 0$ all over $X$. Moreover, if $I$ attains minimum $0$  in the interior $\mathring{X}$, then $I$ is identically 0 all over $X$. Next, we show that it can not be true if $I\equiv 0$ on $X$. Otherwise, $1-|d\phi|^2_{\phi}\equiv 0$ on $X$. For all sufficiently large $t>0$, integrating by parts shows that
$$
0\equiv \int_X e^{-t\phi}(1-|d\phi|^2_{\phi}) \mathrm{dvol}_{\phi}=\int_Xe^{-t\phi} [1-(t-n-1)\phi]  \mathrm{dvol}_{\phi}\triangleeq F(t). 
$$
Taking $t\rightarrow\infty$, we have  $\mu(\{\phi\leq  0\})=0$ where $\mu$ is the measure defined by the volume form
$$
\mathrm{dvol}_{\phi} =(2w_{\phi})^{n+1}=(\sqrt{-1})^{n+1}\det H(\phi) dz^1\wedge d\bar{z}^1 \wedge \cdots \wedge dz^{n+1}\wedge d\bar{z}^{n+1}.
$$ 
Then choose $t_0$ sufficiently large and 
$$
0=\int_{t_0}^{\infty} F(t)dt= -\int_X e^{-t_0\phi} (t_0-n-1) \mathrm{dvol}_{\phi}. 
$$
Hence $\phi \equiv +\infty$ on $X$. However, it can not be true and we finish the proof. 
\end{proof}

Notice that if $R_{\theta}$ is strictly positive, $I>0$ is proved in \cite{LW}. 

\begin{lemma}\label{lem.3.3}
Assume (A1)-(A4) and  the Webster scalar curvature $R_{\theta}\geq 0$ on $M$. 
\begin{itemize}
\item[(a)] If $R_{\theta}>0$ on $M$,  then $\rho$ is plurisubharmonic all over $X$. 
\item[(b)] If $R_{\theta}\geq 0$  on $M$, then $\rho$ is plurisubharmonic in $\mathring{X}$. 
\end{itemize}
\end{lemma}
\begin{proof} Part (a) is proved by Li and Wang \cite{LW}. In local holomorphic coordinates $\{z^i: i=1,...,n+1\}$, recall that
$$
\rho_{i\bar{j}}=(-\rho)(\phi_{i\bar{j}}-\phi_i\phi_j). 
$$
By assumption (A2), $H(\phi)$ is positive definite, so $H(\rho)$ is positive definite if and only if $\det H(\rho)>0$. By direct computation,
$$
\det H(\rho)=(-\rho)^{n+1}(1-|d\phi|^2_{\phi}) \det H(\phi)=e^{\phi}\left(1-|d\phi|^2_{\phi}\right)J[\rho]. 
$$
Using Lemma \ref{lem.3.2}, we prove both (a) and (b). 
\end{proof} 
%

\subsection{Test Functions} For simplicity, in the following of the paper, we always denote
$$
I=e^{\phi}\left(1-|d\phi|^2_{\phi}\right) =\frac{1}{|d\rho|^2_{\rho}-\rho}. 
$$

Define
\begin{equation}\label{test.1}
v_s=(-\rho)^{n+1-s}=e^{-(n+1-s)\phi}. 
\end{equation}

\begin{lemma}\label{lem.3.4}
Assume (A1)-(A4). 
If the Webster scalar curvature $R_{\theta}\geq 0$  on $M$, then in $\mathring{X}$
$$
\frac{\triangle_{\phi}v_s}{v_s}-s(n+1-s)=(n+1-s)^2(-\rho)I>0
$$
for all $s\in \mathbb{R}\backslash \{n+1\}$. 
\end{lemma}
\begin{proof}
Recall that in local holomorphic coordinates $\{z^i: i=1,...,n+1\}$, $\triangle_{\phi}v=-\phi^{i\bar{j}}v_{i\bar{j}}$. 
Hence
$$
\begin{aligned}
&\ \triangle_{\phi}v_s-s(n+1-s)v_s\\
=&\ 
-\phi^{i\bar{j}} e^{-(n+1-s)\phi}\left(-(n+1-s)\phi_{i\bar{j}}+(n+1-s)^2\phi_i\phi_{\bar{j}}\right)
-s(n+1-s)e^{-(n+1-s)\phi}
\\
=&\ (n+1-s)^2e^{-(n+1-s)\phi}\left(1-|d\phi|^2_{\phi}\right)>0
\end{aligned}
$$
in $\mathring{X}$ by Lemma \ref{lem.3.2}.
\end{proof}

Define
\begin{equation}\label{test.2}
w_K=(-\rho)^{\frac{n+1}{2}} \left[ -\log(-\rho)+K\right]
=e^{-\frac{n+1}{2}\phi}\left(\phi+K\right). 
\end{equation}

\begin{lemma}\label{lem.3.5}
Assume (A1)-(A4). 
If the Webster scalar curvature $R_{\theta}\geq 0$ on $M$, then for sufficiently large $K>0$, $w_K$ is positive and in $\mathring{X}$,
$$
\frac{\triangle_{\phi}w_K}{w_K}-\frac{(n+1)^2}{4}>0. 
$$
\end{lemma}
\begin{proof}
Direct computation shows that
$$
\triangle_{\phi}w_K-\frac{(n+1)^2}{4}w_K=e^{-\frac{n+1}{2}\phi} (1-|d\phi|^2_{\phi})
\left( \frac{(n+1)^2}{4}\phi-(n+1)+\frac{(n+1)^2}{4}K\right)>0
$$
in $\mathring{X}$ by Lemma \ref{lem.3.2} if $K$ is large enough. 
\end{proof}

\subsection{Spectral and Resolvent. }
We generalise some results of \cite{GQ} by Guillarmou and Qing for Poinca\'{e}-Einstein metric to ACHE metric of Bergman type. 

\begin{lemma}\label{lem.3.6}
Assume (A1)-(A4) and $R_{\theta}>0$ on $M$. 
If $u\in (-\rho)^{\frac{n+1}{2}}C^{2}(X)$ solves 
$$
\triangle_{\phi}u-\frac{(n+1)^2}{4}u=0, 
$$ 
then $u\equiv 0$. 
\end{lemma}
\begin{proof}
Consider the equation for $u/w_K$ for $K>0$ large:
$$
\triangle_{\phi}\left(\frac{u}{w_K}\right)=\left(s(n+1-s)-\frac{\triangle_{\phi}w_K}{w_K}\right)\left(\frac{u}{w_K}\right) +\phi^{i\bar{j}}\left[\left(\frac{u}{w_K}\right)_i(\log w_K)_{\bar{j}}+\left(\frac{u}{w_K}\right)_{\bar{j}}(\log w_K)_{i}\right]. 
$$
Here by assumption
$$
\frac{u}{w_K}\vline_{M}=0. 
$$
Hence applying Lemma \ref{lem.3.5} and maximum principle shows that $u\equiv 0$. 
\end{proof}

\begin{proposition}
Assume (A1)-(A4). 
If the Webster scalar curvature $R_{\theta}\geq 0$  on $M$, then
$$
\sigma(\triangle_{\phi})=\sigma_{ac}(\triangle_{\phi})=\left[(n+1)^2/4,\infty\right). 
$$
Moreover, the resolvent $R(s)=(\triangle_{\phi}-s(n+1-s))^{-1}$ is analytic at $s=\frac{n+1}{2}$. \end{proposition}
\begin{proof}
The first statement is proved in \cite{LW}. 
For any real  $s<\frac{n+1}{2}$,  the test function $0<v_s=(-\rho)^{n+1-s}\in L^2(X,\mathrm{dvol}_{\phi})$. Since $\triangle_{\phi}v_s - s(n+1-s)v_s>0$, it implies that
 there is no $L^2$-eigenvalues less than $s(n+1-s)$. Since $s$ is arbitrary, $\sigma_{pp}(\triangle_{\phi})=\emptyset$.  Hence $\sigma(\triangle_{\phi})=\sigma_{ac}(\triangle_{\phi})=[(n+1)^2/4,\infty)$. The resolvent $R(s)$ has a finite pole of order $\leq 2$ at $\frac{n+1}{2}$, corresponding to solutions solving
 $$
\triangle_{\phi }u-\frac{(n+1)^2}{4}u=0
$$
and belonging to $L^2(X, \mathrm{dvol}_{\phi})$ or $(-\rho)^{\frac{n+1}{2}}C^{2}(X)$. First, it is showed in \cite{EMM} that $\frac{(n+1)^2}{4}\notin \sigma_{pp}(\triangle_{\phi})$, so there is no solution belonging to $L^2(X, \mathrm{dvol}_{\phi})$. Second, by Lemma \ref{lem.3.6}, there is no solution belonging to  $(-\rho)^{\frac{n+1}{2}}C^{2}(X)$. Hence $R(s)$ is analytic at $s=\frac{n+1}{2}$.
\end{proof}

\subsection{Scattering}

\begin{proposition}\label{prop.3.1}
Assume (A1)-(A4). 
For $\gamma\in(0,1)$ and at any $p\in M$, 
\begin{itemize}
\item[(a)] $R_{\theta}(p)>0$ implies $Q_{2\gamma}(p)>0$;  
\item[(b)] $R_{\theta}(p)\geq 0$ implies $Q_{2\gamma}(p)\geq 0$. 
\end{itemize}
\end{proposition}
\begin{proof} Denote  $s=\frac{n+1+\gamma}{2}$. 
Let $w$ solves  the equation 
$$
\triangle_{\phi}w-s(n+1-s)w=0, \quad w\sim (-\rho)^{n+1-s}  \textrm{ as $\rho\rightarrow 0$}. 
$$
Then using the asymptotic expansion (\ref{laplaceasymp.2}) of $\triangle_{\phi}$, we can compute the asymptotical expansion of $w$ as follows:
$$
w=(-\rho)^{n+1-s} \left[1+(-\rho)^{\gamma}S(s)1+(-\rho)w_1 +\mathcal{O}((-\rho)^{1+\gamma})\right], 
$$
where
$$
w_1=\frac{-1}{2s-n-2}\frac{(n+1-s)^2}{n(n+1)}R_{\theta}.
$$
Compare $w$ with $v_{s}$.  Then $w/v_s$ satisfies
\begin{equation}\label{compareeq.1}
\triangle_{\phi}\left(\frac{w}{v_s}\right)=\left(s(n+1-s)-\frac{\triangle_{\phi}v_s}{v_s}\right)\left(\frac{w}{v_s}\right) +\phi^{i\bar{j}}\left[\left(\frac{w}{v_s}\right)_i(\log v_s)_{\bar{j}}+\left(\frac{w}{v_s}\right)_{\bar{j}}(\log v_s)_{i}\right]. 
\end{equation}
At the boundary
$$
\left(\frac{w}{v_s}\right)\vline_{M}=1. 
$$
By maximum principle, $w\leq v_s=(-\rho)^{n+1-s}$ all over $X$. So near boundary
$$
1+(-\rho)^{\gamma}S(s)1+(-\rho)w_1 +\mathcal{O}((-\rho)^{1+\gamma})\leq 1. 
$$
If $w_1(p)>0$, then it forces $S(s)1<0$ at $p$, which is equivalent to  $Q_{2\gamma}(p)>0$ by the renormalisation formula (\ref{rescatteringop}) for $\gamma\in(0,1)$. Similarly, if $w_1(p)\geq 0$, then $S(s)1\leq 0$, i.e. $Q_{2\gamma}(p)\geq 0$. 
\end{proof}

\begin{proposition}\label{prop.3.2}
Assume (A1)-(A4). 
If the Webster scalar curvature $R_{\theta}>0$  on $M$, then for $\gamma\in(0,1)$ and $f\in C^{\infty}(M)$, we have: 
\begin{itemize}
\item[(a)] if $P_{2\gamma}f>0$  then $f>0$;
\item[(b)] if $P_{2\gamma}f\geq 0$  then $f\geq 0$. 
\end{itemize}
\end{proposition}
\begin{proof}
Denote $s=\frac{n+1+\gamma}{2}$. 
Let $u$ solves  the equation 
$$
\triangle_{\phi}u-s(n+1-s)u=0, \quad u\sim (-\rho)^{n+1-s}f  \textrm{ as $\rho\rightarrow 0$}. 
$$
Then near boundary, $u$ has asymptotical expansion:
$$
u=(-\rho)^{n+1-s} \left[f+(-\rho)^{\gamma}S(s)f+(-\rho)u_1 +\mathcal{O}((-\rho)^{1+\gamma})\right]
$$
where
$$
S(s)f=2^{-2\gamma}\frac{\Gamma(-\gamma)}{\Gamma(\gamma)} P_{2\gamma}f, 
$$
$$
u_1=\frac{-1}{2s-n-2}\left(\frac{1}{2}\triangle_b+\frac{(n+1-s)^2}{n(n+1)}R_{\theta}\right)f. 
$$
Similar as the proof of Proposition \ref{prop.3.1}, $u/v_{s}$ satisfies
\begin{equation}\label{compareeq.2}
\triangle_{\phi}\left(\frac{u}{v_s}\right)=\left(s(n+1-s)-\frac{\triangle_{\phi}v_s}{v_s}\right)\left(\frac{u}{v_s}\right) +\phi^{i\bar{j}}\left[\left(\frac{u}{v_s}\right)_i(\log v_s)_{\bar{j}}+\left(\frac{u}{v_s}\right)_{\bar{j}}(\log v_s)_{i}\right]. 
\end{equation}

First, assume $P_{2\gamma}f>0$. This implies $S(s)f<0$. If $\min_M f=f(p)\leq 0$, then $u_1(p)\leq 0$ and on the boundary
$$
\min_M \left(\frac{u}{v_s}\right)=f(p)\leq 0. 
$$
Applying maximum principle to (\ref{compareeq.2}) shows that there is no interior negative minimum. Hence $u\geq f(p)v_{n+1-s}=f(p)(-\rho)^{n+1-s}$ all over $X$. So at $(p,\rho)$ for $\rho\sim 0$, 
$$
f(p)+(-\rho)^{\gamma}S(s)f|_{p}+(-\rho)u_1(p) +\mathcal{O}((-\rho)^{1+\gamma})\geq f(p).
$$
This conflicts with $S(s)f<0$ and $u_1(p)\leq 0$. So $f>0$ on $M$. 

Second, assume $P_{2\gamma}f\geq 0$. This implies $S(s)f\leq 0$. 
If $\min_M f=f(p)< 0$, then $u_1(p)< 0$ and on the boundary
$$
\min_M \left(\frac{u}{v_s}\right)=f(p)< 0. 
$$
Applying maximum principle to (\ref{compareeq.2}) shows that there is no interior negative minimum. Hence $u\geq f(p)v_s=f(p)(-\rho)^{n+1-s}$ all over $X$. So at $(p,\rho)$ for $\rho\sim 0$,
$$
f(p)+(-\rho)^{\gamma}S(s)f|_{p}+(-\rho)u_1(p) +\mathcal{O}((-\rho)^{1+\gamma})\geq f(p). 
$$
This conflicts with $S(s)f\leq 0$ and $u_1(p)< 0$.  So $f\geq 0$ on $M$. 

\end{proof}

\begin{proposition}\label{prop.3.4}
Assume (A1)-(A4). 
If the Webster scalar curvature $R_{\theta}>0$  on $M$, then for $\gamma\in(0,1)$, the bottom spectrum of $P_{2\gamma}$ is positive. 
\end{proposition}
\begin{proof}
First the bottom spectrum of $P_2$ is positive since $R_{\theta}>0$. Second we show that $P_{2\gamma}$ has no zero spectrum. By \cite{Po} and \cite{GS}, $P_{2\gamma}$ is hypoelliptic and Fredholm on suitable function spaces. Hence zero spectrum implies zero eigenvalue. Assume $P_{2\gamma}f=0$. Then $f\in C^{\infty}(M)$. By part (b) of Proposition \ref{prop.3.2}, $P_{2\gamma}f\geq 0$ implies $f\geq 0$ and $P_{2\gamma}(-f)\geq 0$ implies $-f\geq 0 $. Therefore $f=0$. Finally by continuity argument,  the bottom spectrum of $P_{2\gamma}$ is positive for all $\gamma\in(0,1)$. 
\end{proof}

A similar proof as Proposition \ref{prop.3.2} shows that
\begin{proposition} 
Assume (A1)-(A4). 
If the Webster scalar curvature $R_{\theta}\geq 0$  on $M$, then for $\gamma\in(0,1)$ and $f\in C^{\infty}(M)$,  
\begin{itemize}
\item[(a)] if $P_{\gamma}f>0$ on $M$ then $f>0$ on $M$;
\item[(b)] if $P_{\gamma}f\geq 0$ on $M$  then either $f\geq 0$ on $M$ or 
$$
\min_M f=f(p)<0, \quad P_{\gamma}f|_p=0, \quad R_{\theta}(p)=0. 
$$
\end{itemize}
\end{proposition}
%

%
%

%
%

\begin{proposition}
Assume (A1)-(A4). 
If the Webster scalar curvature $R_{\theta}\geq 0$ on $M$, then
$$
S\left(\frac{n+1}{2}\right)=-Id. 
$$
\end{proposition}
\begin{proof}
Since the resolvent $R(s)$ is analytic at $s=\frac{n+1}{2}$, then $S(s)$ ia analytic at $s=\frac{n+1}{2}$. For $s=\frac{n+1}{2}+\sqrt{-1} \sigma$ with
 $\sigma\neq 0$, we have 
$$
S^*(s)=S(\bar{s})=S(n-s), \quad S(s)S(n-s)=Id. 
$$
Hence by the meromorphic extension,
$$
\left[S\left(\frac{n+1}{2}\right)\right]^2=Id. 
$$
Finally, by Lemma \ref{lem.3.6} and self-adjoint property of $S(\frac{n+1}{2})$,  we have $S(\frac{n+1}{2})=-Id$. 
\end{proof}

\vspace{0.2in}
\section{Energy Identity}

Throughout of this section, we fix $s=\frac{n+1+\gamma}{2}$ with $\gamma\in(0,1)$. Assume (A1)-(A4) and
\begin{itemize}
\item[\textbf{(A5)}] The boundary CR structure $(M,J,\theta)$ has positive Webster scalar curvature, i.e.  $R_{\theta}>0$ on $M$. 
\end{itemize}
Hence by Lemma \ref{lem.3.3}, $\rho$ is strictly plurisubharmonic all over $X$. Now we have two K\"{a}hler metrics $g_{\phi}$ and $g_{\rho}$,  associated with K\"{a}hler forms $\omega_{\phi}$ and $\omega_{\rho}$ respectively. In local holomorphic coordinates, 
$$
\omega_{\rho}=\frac{\sqrt{-1}}{2} \rho_{i\bar{j}} dz^i\wedge d\bar{z}^j, \quad
\omega_{\phi}=\frac{\sqrt{-1}}{2} \phi_{i\bar{j}} dz^i\wedge d\bar{z}^j. 
$$
In our convention, we take 
$$
\begin{gathered}
\triangle_{\phi} u=-\phi^{i\bar{j}}u_{i\bar{j}}, \quad
\triangle_{\rho} u=-\rho^{i\bar{j}}u_{i\bar{j}}
\\
\mathrm{dvol}_{\phi} =(2\omega_{\phi})^{n+1}, \quad
\mathrm{dvol}_{\rho} =(2\omega_{\rho})^{n+1}.
\end{gathered}
$$ 
More over, near the boundary we have a second set of frame 
 $\{W_{\alpha}, W_{\bar{\alpha}}, \Xi, \overline{\Xi}: \alpha=1,..., n\}$ with coframe $\{\theta^{\alpha}, \theta^{\bar{\alpha}}, \partial\rho, \overline{\partial}\rho: \alpha=1,..., n\}$, which are compatible with the boundary CR-structure on each level set $M_{\epsilon}$, such that
$$
\begin{aligned}
&\omega_{\rho}= \frac{\sqrt{-1}}{2}\left(h_{\alpha\bar{\beta}}\theta^{\alpha}\wedge \theta^{\bar{\beta}}
+r\partial \rho\wedge \overline{\partial}\rho\right), \quad
\omega_{\phi}= \frac{\sqrt{-1}}{2}\left(\frac{1}{-\rho} h_{\alpha\bar{\beta}}\theta^{\alpha}\wedge \theta^{\bar{\beta}}+
\frac{1-r\rho}{\rho^2}\partial \rho\wedge \overline{\partial}\rho\right). 
\end{aligned}
$$

We defined three functions:
\begin{itemize}
\item[(a)] Let $v=(-\rho)^{n+1-s}$. Then by Lemma \ref{lem.3.4}
$$
\begin{gathered}
\frac{\triangle_{\phi}v}{v}-s(n+1-s)=(n+1-s)^2(-\rho)I>0, 
\quad\textrm{where}
\\
I=e^{\phi}(1-|d\phi|^2_{\phi})=\frac{1}{|d\rho|^2_{\rho}-\rho}>0, \quad
I=\frac{r}{1-r\rho}\ \textrm{for $\rho\sim 0$. }
\end{gathered}
$$

\item[(b)]  Let $w \sim (-\rho)^{n+1-s}$ as $\rho\rightarrow 0$ and satisfy $\triangle_{\phi}w-s(n+1-s)w=0$. Then
$$
\begin{gathered}
w=(-\rho)^{n+1-s} \left[1+(-\rho)^{\gamma}w_{\gamma}+(-\rho)w_1 +\mathcal{O}((-\rho)^{1+\gamma})\right], 
\quad\textrm{where}
\\
w_{\gamma}=S(s)1=2^{-2\gamma}\frac{\Gamma(-\gamma)}{\Gamma(\gamma)}Q_{2\gamma}, \quad
w_1=\frac{-1}{2s-n-2}\frac{(n+1-s)^2}{n(n+1)}R_{\theta}. 
\end{gathered}
$$
Moreover, $w>0$ since $\mathrm{spec}(\triangle_{\phi})=[(n+1)^2/4, \infty)$ by Proposition \ref{prop.3.1}. 
\item[(c)] Let $u\sim  (-\rho)^{n+1-s}f$ as $\rho\rightarrow 0$ and satisfy $\triangle_{\phi}u-s(n+1-s)u=0$. Then
$$
\begin{gathered}
u=(-\rho)^{n+1-s} \left[f+(-\rho)^{\gamma}u_{\gamma}+(-\rho)u_1 +\mathcal{O}((-\rho)^{1+\gamma})\right]
\quad\textrm{where}
\\
u_{\gamma}=S(s)f=2^{-2\gamma}\frac{\Gamma(-\gamma)}{\Gamma(\gamma)} P_{2\gamma}f, 
\quad
u_1=\frac{-1}{2s-n-2}\left(\frac{1}{2}\triangle_b+\frac{(n+1-s)^2}{n(n+1)}R_{\theta}\right)f. 
\end{gathered}
$$
\end{itemize}
\begin{lemma}\label{lem.4.1}
Assume (A1)-(A5) and 
let $u=vU$. Then
$$
\begin{cases}
\triangle_{\rho}U+I 
\left[\rho^i\rho^{\bar{j}}U_{i\bar{j}}+(n+1-s)(\rho^iU_i+\rho^{\bar{j}}U_{\bar{j}})+(n+1-s)^2U\right]
=0, 
\\
(-\rho)^{1-\gamma} N(U)|_{M}=-2\gamma u_{\gamma}, 
\end{cases}
$$
where $N=\Xi+\overline{\Xi}$ satisfies $d\rho(N)=2$. 
\end{lemma}
\begin{proof}
Direct computation shows that
$$
\triangle_{\phi}u-s(n+1-s)u=\rho^{n+2-s}\left(
\triangle_{\rho}U+I 
\left[\rho^i\rho^{\bar{j}}U_{i\bar{j}}+(n+1-s)(\rho^iU_i+\rho^{\bar{j}}U_{\bar{j}})+(n+1-s)^2U\right]
\right). 
$$
Using the asymptotic expansion of $U$, we pick out the global term $u_{\gamma}$ as given above. 
\end{proof}
\begin{proposition} \label{prop.4.1}
Assume (A1)-(A5) and 
let $u=vU$. Then
$$
d_{\gamma}\oint_{M}fP_{2\gamma}f   \theta \wedge (d\theta)^n
=\int_X  (-\rho)^{-\gamma} \left\{\left[ (|d\rho|^2_{\rho}-\rho)\rho^{i\bar{j}}
-\rho^i\rho^{\bar{j}}\right] U_iU_{\bar{j}}+ (n+1-s)^2|U|^2 \right\}\mathrm{dvol}_{\rho}\geq 0,  
$$
where 
$$
d_{\gamma}=-\gamma 2^{-2\gamma}\frac{\Gamma(-\gamma)}{\Gamma(\gamma)}>0.
$$
\end{proposition}
\begin{proof}
In local holomorphic coordinates, 
$$
\begin{aligned}
\mathrm{dvol}_{\phi} =&\ 
(\sqrt{-1})^{n+1}\det H(\phi) dz^1\wedge d\bar{z}^1\wedge \cdots \wedge dz^{n+1}\wedge d\bar{z}^{n+1}, 
\\
\mathrm{dvol}_{\rho}=&\  
(\sqrt{-1})^{n+1}\det H(\rho) dz^1\wedge d\bar{z}^1\wedge \cdots \wedge dz^{n+1}\wedge d\bar{z}^{n+1}, 
\end{aligned}
$$
$$
\mathrm{dvol}_{\phi} =\frac{\det H(\phi) }{\det H(\rho) } \mathrm{dvol}_{\rho} =\frac{1}{I(-\rho)^{(n+2)}}
\mathrm{dvol}_{\rho}. 
$$
By Lemma \ref{lem.4.1} and direct computation 
$$
\begin{aligned}
&\  \int_X u(\triangle_{\phi}u-s(n+1-s)u) \mathrm{dvol}_{\phi} \\
=&\ \int_X (-\rho)^{-\gamma} U \left[ (|d\rho|^2_{\rho}-\rho)\triangle_{\rho}U +\rho^i\rho^{\bar{j}} U_{i\bar{j}}+(n+1-s)(\rho^iU_i+\rho^{\bar{j}}U_{\bar{j}})+(n+1-s)^2U
\right] \mathrm{dvol}_{\rho}. 
\end{aligned}
$$
The volume form near boundary can be represented by
$$
\begin{aligned}
\mathrm{dvol}_{\rho}=&\  (\sqrt{-1})^{n+1} r\det [h_{\alpha\bar{\beta}}] \partial\rho \wedge \overline{\partial}\rho \wedge \theta^1\wedge \theta^{\bar{1}}\wedge \cdots \wedge \theta^n\wedge \theta^n
\\ 
=&\   r d\rho\wedge \Theta \wedge \left(\sqrt{-1})^{n} \det [h_{\alpha\bar{\beta}}]  \theta^1\wedge \theta^{\bar{1}}\wedge \cdots \wedge \theta^n\wedge \theta^n \right)
\\
=&\ rd\rho \wedge \theta \wedge (d\theta)^n. 
\end{aligned}
$$
Here $\Theta=\frac{\sqrt{-1}}{2}(\overline{\partial}\rho-\partial\rho)$ and $\theta=\theta_{\epsilon}=i^*_\epsilon \Theta $ with $i_{\epsilon}: M_{\epsilon}\rightarrow X$. 
Integrating by parts shows
$$
\begin{aligned}
&\  \int_X u(\triangle_{\phi}u-s(n+1-s)u) \mathrm{dvol}_{\phi} \\
=&\ \int_X  (-\rho)^{-\gamma} \left\{\left[ (|d\rho|^2_{\rho}-\rho)\rho^{i\bar{j}}
-\rho^i\rho^{\bar{j}}\right] U_iU_{\bar{j}} + (n+1-s)^2U^2 \right\}\mathrm{dvol}_{\rho}
+\oint_{M}\gamma u_{\gamma}f   \theta \wedge (d\theta)^n .
\end{aligned}
$$
Notice that $(|d\rho|^2_{\rho}-\rho)\rho^{i\bar{j}}
-\rho^i\rho^{\bar{j}} =(-\rho)^{-1}I^{-1}\phi^{i\bar{j}}$ is positive definite. Hence
$$
\left[ (|d\rho|^2_{\rho}-\rho)\rho^{i\bar{j}}
-\rho^i\rho^{\bar{j}}\right] U_iU_{\bar{j}} \geq 0.
$$
Recall $\triangle_{\phi}u-s(n+1-s)u=0$. We finish the proof. 
\end{proof}

Similar as the energy identity for adapted metric measure space in \cite{CC}, we have a second energy identity here such that the interior energy has no zero order term. 

\begin{lemma}\label{lem.4.2}
Assume (A1)-(A5). 
Let $w=vW$ and $u=wU'=vWU'$. Then
$$
\begin{cases}
\begin{aligned}
W\triangle_{\rho} U' &-\rho^{i\bar{j}}(W_iU'_{\bar{j}}+W_{\bar{j}}U'_i)
\\
&+ I\left[
W \rho^i\rho^{\bar{j}}U'_{i\bar{j}} +\rho^i\rho^{\bar{j}}(W_iU'_{\bar{j}}+W_{\bar{j}}U'_i)
+(n+1-s)(\rho^iU'_i+\rho^{\bar{j}}U'_{\bar{j}}) W\right]=0, 
\end{aligned}
\\
(-\rho)^{1-\gamma}N(U')|_{M}=-2\gamma(u_{\gamma}-w_{\gamma}f). 
\end{cases}
$$
\end{lemma}
\begin{proof} By direct computation
$$
\begin{aligned}
\triangle_{\phi}u-s(n+1-s)u=&\ 
\rho^{n+2-s}\left(
W\triangle_{\rho} U' -\rho^{i\bar{j}}(W_iU'_{\bar{j}}+W_{\bar{j}}U'_i)\right)+
\\
&\ \rho^{n+2-s}I\left[
W \rho^i\rho^{\bar{j}}U'_{i\bar{j}} +\rho^i\rho^{\bar{j}}(W_iU'_{\bar{j}}+W_{\bar{j}}U'_i)
+(n+1-s)(\rho^iU'_i+\rho^{\bar{j}}U'_{\bar{j}}) W
\right].
\end{aligned}
$$
From the asymptotical expansion of $u$ and $w$, we can calculate  the global term in $U'$. 
\end{proof}
\begin{proposition}\label{prop.4.2}
Assume (A1)-(A5). 
Let $w=vW$ and $u=wU'=vWU'$. Then
$$
d_{\gamma}\oint_{M}(fP_{\gamma}f -Q_{\gamma}f^2 ) \theta \wedge (d\theta)^n
=\int_X  (-\rho)^{-\gamma}W^2\left[ (|d\rho|^2_{\rho}-\rho)\rho^{i\bar{j}}
-\rho^i\rho^{\bar{j}}\right] U'_iU'_{\bar{j}} \mathrm{dvol}_{\rho} \geq 0, 
$$
where $d_{\gamma}$ is the same as in Proposition \ref{prop.4.1}. 
\end{proposition}

\begin{proof} A similar computation as in the proof of Proposition \ref{prop.4.1} shows
$$
\begin{aligned}
&\  \int_X u(\triangle_{\phi}u-s(n+1-s)u) \mathrm{dvol}_{\phi} \\
=&\ \int_X  (-\rho)^{-\gamma} W^2 \left[ (|d\rho|^2_{\rho}-\rho)\rho^{i\bar{j}}
-\rho^i\rho^{\bar{j}}\right] U'_iU'_{\bar{j}}  \mathrm{dvol}_{\rho}
+\oint_{M}\gamma (u_{\gamma}-w_{\gamma}f)f   \theta \wedge (d\theta)^n. 
\end{aligned}
$$
\end{proof}

\begin{corollary}\label{cor.4.1}
Assume (A1)-(A5). 
For $f\in C^{\infty}(M)$, 
$$
\oint_M f P_{2\gamma}f \theta\wedge (d\theta)^n \geq \oint_M  Q_{2\gamma}f^2 \theta\wedge (d\theta)^n.
$$
\end{corollary}

\begin{remark}
In the energy identities in Proposition \ref{prop.4.1} and Proposition \ref{prop.4.2}, near  boundary
$$
\left[ (|d\rho|^2_{\rho}-\rho)\rho^{i\bar{j}}
-\rho^i\rho^{\bar{j}}\right] U_iU_{\bar{j}} 
=\mathcal{O}
\left(
 (-\rho)|N(U)|^2 +(-\rho) |T(U)|^2 
+ h^{\alpha\bar{\beta}}  U_{\alpha}U_{\bar{\beta}}\right).
$$
Here $(-\rho)|N(U)|^2=|\partial_rU|^2$ if we take $(-\rho)=r^2$. The degeneration in characteristic direction $T$, which corresponds to the non compactness of the metric
$$
G=G_{i\bar{j}} dz^i\varodot d\bar{z}^j
=\frac{1}{2}\left(\rho_{i\bar{j}}+\frac{\rho_i\rho_{\bar{j}}}{-\rho}\right) dz^i\varodot d\bar{z}^j, 
$$
is compatible with the Heisenberg calculus on the boundary. This is a bit different from the real case. 
\end{remark}

\textbf{Proof of Theorem \ref{thm.1}:} Part (a) is from Proposition \ref{prop.3.1}; Part (b) is from Proposition \ref{prop.3.2}; Part (c) is from Proposition \ref{prop.3.4} and Corollary \ref{cor.4.1}. 

\textbf{Proof of Theorem \ref{thm.2}:} Part (a) is from Proposition \ref{prop.4.1}; Part (b)  is from Proposition \ref{prop.4.2}.

\end{document}